\documentclass[english]{amsart}
\usepackage[T1]{fontenc}
\usepackage[latin9]{inputenc}
\usepackage{amsthm}
\usepackage{amssymb}

\makeatletter
\numberwithin{equation}{section}
\numberwithin{figure}{section}
\theoremstyle{plain}
\newtheorem{thm}{\protect\theoremname}[section]
  \theoremstyle{definition}
  \newtheorem{defn}[thm]{\protect\definitionname}
  \theoremstyle{plain}
  \newtheorem{prop}[thm]{\protect\propositionname}
  \theoremstyle{plain}
  \newtheorem{cor}[thm]{\protect\corollaryname}
  \theoremstyle{plain}
  \newtheorem{lem}[thm]{\protect\lemmaname}

\makeatother

\usepackage{babel}
  \providecommand{\corollaryname}{Corollary}
  \providecommand{\definitionname}{Definition}
  \providecommand{\lemmaname}{Lemma}
  \providecommand{\propositionname}{Proposition}
\providecommand{\theoremname}{Theorem}

\begin{document}

\title{A Generalization of Ribet's Nondegeneracy Theorem}

\author{Ryan Keast}

\address{Ryan Keast}

\address{Department of Mathematics, Washington University in St. Louis St.
Louis, MO 63130, USA}

\email{rkeast@wustl.edu}

\thanks{This work was supported by the National Science Foundation {[}DMS-1361147
PI: Matthew Kerr{]} . The author would like to thank Matt Kerr and
Greg Pearlstein for reading previous manuscripts and providing immeasurable
support.}
\begin{abstract}
We extend Ribet's Nondegeneracy theorem to all odd weight Hodge structures.
\end{abstract}
\maketitle

\section{Introduction}
 A CM Hodge structure is a Hodge structure whose Mumford-Tate (MT)
 group is abelian. For any period domain, and any Mumford-Tate domain,
 CM Hodge structures are dense in both the Zariski and Archimidean
 topologies.
 In the case where $D$ has a trivial infinitesimal period relation, the Andre-Oort conjecture is essentially a converse to this statement.  Let $D$ be
 an unconstrained Mumford-Tate domain and $\Gamma\backslash D$ be
 a connected Shimura variety. Andre-Oort states that the Zariski closure
 of a subset of CM points in $\Gamma\backslash D$ is itself a Shimura 
 subvariety.
 
 Recently there has been considerable progress in the understanding of non-classical period domains.  Specifically, the work of Griffiths, Green, and Kerr \cite{GGK} has extended the notion of CM-type to a higher weight analogue, called an orientation. To each irreducible CM Hodge structure there is an associated oriented CM field. Using the Galois groups of these CM fields and the density of CM points, \cite{GGK} developed an algorithm for classifying the Mumford-Tate groups for a given period domain.
 Even though this paper does not directly address it, the primary motivation behind our efforts
 is to explore possible generalizations of Andre-Oort and its related
 conjectures to non-classical period domains: what does the Zariski density of CM points imply for a variation of Hodge structure in the higher weight case?
 
 In light of this goal, this paper studies irreducible CM points and Mumford-Tate domains
 contained in higher odd weight period domains.
 
Our first and main result concerns the Mumford-Tate groups of odd
weight polarized irreducible CM Hodge structures. Assume the Hodge
numbers add up to $2p$, where $p$ is an odd prime. Ribet's nondegeneracy
theorem states that in the weight one case the Mumford-Tate group
of such a CM point is a maximal compact torus. We extend this to all
odd weights.
 
  Ribet's proof involves calculating the dimension of a subspace generated by the Galois group acting on a grading element.  The grading element is dependent on the weight and Hodge numbers, so the proof does not immediately extend to higher weights. We sidestep this difficulty with some knowledge of circulant matrices.     
  
  In contrast to the main result, the rest of the paper shows the radical differences that can show up in the higher weight case.  Many of the proofs here use Dodson's Imprimitivity Theorem as inspiration\cite{D}.
  
  Schmid's monodromy theorem
  informs us that in the weight one case, a nilpotent matrix $N$ arising
  from a degeneration of a variation of Hodge structure must satisfy
  $N^{2}=0$ \cite{S}. In the higher weight cases, $N$ with larger Jordan blocks
  may appear. We show that if a Jordan blocks is sufficiently large and the nilpotent orbit passes through a nondegenerate CM point,
  the smallest MT domain than contains the nilpotent orbit is in fact the whole period domain.

Concluding, we provide examples of Calabi-Yau CM points which are never contained
in positive dimension unconstrained MT domains.

\section{Review}

Let $D$ be a period domain with fixed rational vector space $V$
and polarization $Q$.  We construct the following algebraic group:

\begin{equation}
G=\{g\in Aut(V)|Q(gv,gu)=Q(u,v)\:\forall u,v\in V\}.
\end{equation}
We restrict ourselves to the case where the weight is odd so that
we may identify $G$ with $Sp(2n,Q)$. $G_{\mathbb{R}}$ acts transitively
on $D$ and we may identify $D=G_{\mathbb{R}}/H_{\varphi}$, where
$H_{\varphi}$ is the subgroup that fixes some $\varphi\in D$. Similarly
$G_{\mathbb{C}}$ acts transitively on the compact dual $\check{D}$. Consider $G_\mathbb{C}$'s associated Lie algebra: 

\begin{equation}
\mathfrak{g}_{\mathbb{C}}=\{X\in End(V_{\mathbb{C}})|Q(Xu,v)+Q(u,Xv)=0\:\forall u,v\in V\}.\label{eq:-1}
\end{equation}
For a point $\varphi\in D$, $\mathfrak{g}_{\mathbb{C}}$ inherits
a weight $0$ Hodge decomposition:
\begin{equation}
\mathfrak{g}^{l,-l}=\{X\in\mathfrak{g}|X(V^{p,q})\subseteq V^{p+l,q-l}\}.\label{eq:}
\end{equation}

A subvariety of $D$ will be called horizontal if its tangent bundle
lives in the subbundle generated by the orbit of $\mathfrak{g}^{-1,1}$.
From Griffiths's transversality, a subvariety of $D$ that arises from
a variation of polarized Hodge structures will always be horizontal.  A MT domain is considered unconstrained if it is horizontal.

Regarding a Hodge structure $\varphi\in D$ as a real representation of
$\mathbb{C}^{*}$, the first Hodge-Riemann relation indicates $\varphi(S^{1})\subset G_{\mathbb{R}}$.
The smallest $\mathbb{Q}$-algebraic group whose real points contain
the image of the circle $\varphi(S^{1})$ is called the Mumford-Tate
group $M_{\varphi}$ of $\varphi$. 
\begin{defn}
A weight $n$ oriented CM field is a pair 
\begin{equation}
{\textstyle \left(L,\left\{ \prod^{p,q}\right\} _{p+q=n}\right)}
\end{equation}
where $L$ is a CM field and $\left\{ \prod^{p,q}\right\} _{p+q=n}$
is a partition of the set of embeddings of $L$ into $\mathbb{C}$
with the following property: 
\begin{equation}
{\textstyle \theta\in\prod^{p,q}\Leftrightarrow\bar{\theta}\in\prod^{q,p}}.
\end{equation}
\end{defn}

To obtain the Hodge structure, we view the CM field $L$ itself as
a $\mathbb{Q}$ vector space. An element $l\in L$ acts on $L$ by
left multiplication, the action decomposes $L\otimes\mathbb{C}$ into
one dimensional $\theta_{k}(l)$ eigenspaces we label $E_{\theta_{k}}$.
We gain a Hodge decomposition by setting
\begin{equation}
V^{p,q}=\oplus_{\theta_{k}\in\prod^{p,q}}E_{\theta_{k}}.
\end{equation}

\begin{thm}[cf. \cite{GGK} pg 196]
Every Hodge structure
constructed in the previous manner is a polarizable CM Hodge structure
{[}not necessarily irreducible{]} and every odd weight irreducible
polarized CM Hodge structure comes from such a construction. 
\end{thm}
An irreducible CM Hodge structure $\varphi\in D$ will then give an
arithmetically defined basis of $V_{\mathbb{C}}$ that respects the
Hodge structure.
\begin{defn}
Let $L$ be a CM field regarded as $\mathbb{Q}$-vector space. The
Hodge-Galois basis of $L\otimes_{\mathbb{Q}}\mathbb{C}$ is the basis
$\{w_{k}\}$ with $w_{k}\in E_{\theta_{k}}$, scaled so that for $l\in L$,
$l=\sum\theta_{k}(l)w_{k}$ when considered as an element in $L\otimes_{\mathbb{Q}}\mathbb{C}$. 
\end{defn}
Via the action on the embeddings, $\mathrm{Gal}(L^{c}/\mathbb{Q})$ acts on
the Hodge-Galois basis. Since $L^{c}$ is Galois, every embedding
has the same image. We fix some embedding $\Theta:\: L^{c}\rightarrow\mathbb{C}$.
For $a\in\Theta(L^{c})$ , $a$= $\Theta(l)$ for some $l\in L^{c}$.
We define for $g\in \mathrm{Gal}(L^{c}/\mathbb{Q})$, $g(a)=\Theta(gl).$ For
$a_{k}\in\Theta(L^{c})$, we have

\begin{equation}
g\left(\sum a_{k}w\right)=\sum g(a_{k})g(w_{k}).
\end{equation}

By the construction of the Hodge-Galois basis, $\sum a_{k}w$ is in
the original rational vector space $L$ if and only if it is fixed
by $\mathrm{Gal}(L^{c}/\mathbb{Q}).$

A polarization of a CM Hodge structure will be called a Hodge-Galois
polarization if the following condition is met:

\begin{equation}
Q(w_{l},\bar{w}_{k})=0\mbox{ unless \ensuremath{l=k}}.
\end{equation}

This condition makes it much easier to explicitly handle $\mathfrak{g}_{\mathbb{C}}$.
Luckily, for irreducible CM Hodge structures this condition is trivial.
\begin{prop}[\cite{GGK} pg. 199]
Every polarization of an irreducible CM Hodge structure
is a Hodge-Galois polarization.
\end{prop}

For odd weight period domains with polarization $Q$, $\mathfrak{g}_{\mathbb{C}}\cong\mathfrak{sp}(2n,Q)$.
The Hodge-Galois polarization allows us to explicitly construct an
arithmetically defined Hodge basis of $\mathfrak{g}_{\mathbb{C}}$. 

Let $\hat{w}_{nm}$ indicate the endomorphism that takes $w_{m}$
to $w_{n}$ and kills everything else. Denoting $Q_{i}=Q(w_{i},\bar{w_{i}})$,
we can form a Hodge basis of $\mathfrak{g}_{\mathbb{C}}$
\begin{equation}
\left\{ \hat{w}_{i,j}+\frac{Q_{i}}{-Q_{j}}\hat{w}_{-j,-i}\right\} .
\end{equation}

If the $\alpha_{i,j}$'s are in $\Theta(L^{c})$, $\mathrm{Gal}(L/Q)$ also
has natural action given by 
\begin{equation}
g\left(\sum\alpha_{i,,j}\left(\hat{w}_{i,j}+\frac{Q_{i}}{-Q_{j}}\hat{w}_{-j,-i}\right)\right)=\sum\left(g(\alpha_{i,,j})\hat{w}_{g(i),g(j)}+\frac{Q_{g(i)}}{-Q_{g(j)}}\hat{w}_{g(-j),g(-i)}\right).
\end{equation}

Again we have $\sum\alpha_{i,,j}\left(w_{i,j}+\frac{Q_{i}}{-Q_{j}}w_{-j,-i}\right)$is
in the original rational vector space if and only if it is fixed by
$\mathrm{Gal}(L^{c}/\mathbb{Q})$.

\section{Nondegeneracy}
\begin{defn}
Let $\varphi\in D$ be a polarized CM Hodge structure. We say $\varphi$
is nondegenerate if $\dim(M_{\varphi})=\dim_{\mathbb{Q}}(V)/2$
\end{defn}

The Ribet's nondegeneracy theorem says that if the
weight is one and $\dim(V)=2p$ with $p>2$ and prime, then being irreducible
implies nondegenerate \cite{R2}.
\begin{thm}[The Generalized Ribet Non-degeneration Theorem]
Let $(V,Q,\varphi)$
be an irreducible polarized odd weight CM Hodge structure. If $\dim(V)=2p$
where $p>2$ is prime, then $\dim(M_{\varphi})=p$.\end{thm}
\begin{proof}
Let $F$ be the associated CM field, and let $F^{c}$ be its Galois
closure. Since $p$ is prime and divides $|\mathrm{Gal}(F^{c})|$, basic group
theory informs us that $\mathrm{Gal}(F^{c}/\mathbb{Q})$ contains the cyclic
subgroup $\mathbb{Z}_{p}$. We write the generator of $\mathbb{Z}_{p}$
as $(0,1,2...p-1)$.

Regarding a Hodge structure $\varphi\in D$ as a representation of
$\mathbb{C}^{*}$, the first Hodge-Riemann relation indicates $\varphi(S^{1})\subset G_{\mathbb{R}}$.
Let $M_{\varphi}$ be the special Mumford-Tate group. Let $\mathfrak{m}_{\mathbb{C}}\subseteq\mathfrak{g}_{\mathbb{C}}$
be the associated Lie algebra of $M_{\varphi}(\mathbb{C})$. $\varphi'(1)\in\mathfrak{m}_{\mathbb{C}}$
has the property that given $x\in\mathfrak{g}^{l,-l}$ 
\begin{equation}
[\frac{1}{2}\varphi'(1),x]=lx\label{eq:-2}
\end{equation}
We express the grading element as a sum of the Hodge-Galois basis
of $\mathfrak{g}_{\mathbb{C}}$: $\frac{1}{2}\varphi'(1)=\sum A_{m}\hat{w}_{mm}-\hat{w}{}_{-m-m}$.

Since the weight is odd and $[\frac{1}{2}\varphi'(1),w_{m,-m}]=A_{m}w_{m,-m}$, we can conclude that all the $A_{m}$ are odd
integers. 

Since the $\mathfrak{m}$ must be defined over $\mathbb{Q}$ , it
must contain $g(\varphi'(1))$ for all $g\in \mathrm{Gal}(F^{C}).$ Since it
contains the cyclic subgroup $\mathbb{Z}_{p}$, the dimension of $\mathfrak{m}$
must be at least the size of the rank of the following circulant matrix:
\[
\left(\begin{array}{ccccc}
A_{0} & \cdots & \cdots & \cdots & A_{p-1}\\
A_{1} & A_{2} & \cdots & A_{p-1} & A_{0}\\
\vdots & \ddots &  &  & \vdots\\
\vdots &  & \ddots &  & \vdots\\
A_{p-1} & \cdots & \cdots & \ddots & A_{p-2}
\end{array}\right).
\]

As a general fact of circulant matrices, the rank is given
$p-deg[gcd(f,x^{p}-1)]$ where $f=A_{0}+A_{1}x+\cdots+A_{p-1}x^{p-1}$\cite{I}.
Since all of the coefficients of $f$ are odd numbers, $1$ is not
a root of $f$. $x^{p-1}+x^{p-2}+...+1$ is the minimal polynomial
of the $p^{th}$ roots of unity. If it shared a root with $f$ it
would have to be a rational multiple of $x^{p-1}+x^{p-2}+...+1$.
So either the rank is $p$ or $A_{0}=A_{1}=...=A_{p-1}$. The latter
case recovers a weight one grading element, and is thus then covered
by the original Ribet Nondegeneracy theorem.
\end{proof}

\section{Imprimitive Domains and MT Domains}

Let $\varphi\in D$ be irreducible and CM with associated CM field
$L$. Let $v\in\mathfrak{g}_{\mathbb{Q}}.$ Consider $v$ as a matrix
defined by the ordered Hodge-Galois basis $\left(w_{1},...w_{n},w_{-1}...,w_{-n}\right)$
of $V_{\mathbb{C}}$. We construct a graph in the following manner:
the vertices are defined to be the elements of the Hodge-Galois basis.
Two vertices are connected by an edge if either entry $a_{i,j}$ or
$a_{j,i}$ in $M$ are non-zero. {[}Note: Because it is fixed under
complex conjugation and must respect the Hodge-Galois polarization,
$a_{i,j}$ and $a_{j,i}$ are either both non-zero or both zero.{]}
The connected components of this graph gives us a partition of the
Hodge-Galois basis. 
\begin{defn}
For $v\in\mathfrak{g}_{\mathbb{Q}}$ $\Gamma_{v}$ and $\pi_{v}$
are respectively the graph and partition constructed above. Using
the same set of vertices, for a Mumford-Tate group $M$, we define
the following graph $\Gamma_{M}$: Two vertices in $\Gamma_{M}$ are
connected if and only if there is a $v\in\mathfrak{m}_{\mathbb{Q}}$
where the vertices are connected in $\Gamma_{v}$. In an identical
manner we derive a partition $\pi_{M}$.\end{defn}
\begin{thm}
The partitions $\pi_{v}$ and $\pi_{M}$ both give an imprimitive
system for the $\mathrm{Gal}(L^{c}/\mathbb{Q})$ action on the Hodge basis
$\left(w_{1},...w_{n},w_{-1}...,w_{-n}\right)$.\end{thm}
\begin{proof}
An alteration of the partition would contradict the fact that for
all $v\in\mathfrak{g}_{\mathbb{Q}}$, $v$ is fixed by $g\in \mathrm{Gal}(L^{c}/\mathbb{Q})$.\end{proof}
\begin{cor}
\label{cor:Trivial parition}Let $N\in\mathfrak{m}_{\mathbb{Q}}$
be nilpotent, if $N^{l-1}\neq0$ and $l>[L:\mathbb{Q}]/2$ then the
partition $\pi_{N}$ is trivial.\end{cor}
\begin{proof}
The degree of $N$ is bounded by the largest size of a connected component
$\Gamma_{N}$. Since the group action is transitive, each primitive
domain must have the same number of elements. If a domain of imprimitivity
contains more than half of the elements, it must be the whole set.\end{proof}
\begin{lem}
\textup{\label{lem:Non-degenerate}Assume $\varphi\in D$ is CM and
nondegenerate and odd weight. Let $M$ be a Mumford-Tate group whose
Mumford-Tate domain passes through $\varphi$. Assume $v\in\mathfrak{m}_{\mathbb{C}}$
. We express $v$ using the Hodge-Galois basis of $\mathfrak{g}_{\mathbb{C}}$:
$v=\sum\alpha_{i,,j}\left(\hat{w}_{i,j}+\frac{Q_{i}}{-Q_{j}}\hat{w}_{-j,-i}^{w}\right).$
For each $\alpha_{i,j}\neq0$, $\left(\hat{w}_{i,j}+\frac{Q_{i}}{-Q_{j}}\hat{w}_{-j,-i}\right)\in\mathfrak{m}_{\mathbb{C}}$. }\end{lem}
\begin{proof}
Since $\varphi$ is CM, $M_{\varphi}$ is abelian. Because of nondegeneracy,
$\mathfrak{m}_{\varphi,\mathbb{C}}$ is maximal and Abelian, hence
is a Cartan subalgebra of $\mathfrak{sp}(2n,Q)$. From \cite{GGK}
we know that for $i\neq j$ $\left(\hat{w}_{i,j}+\frac{Q_{i}}{-Q_{j}}\hat{w}_{-j,-i}\right)$
is a root vector. For $i=j$ , $\left(\hat{w}_{i,i}-\hat{w}_{-i,-i}\right)$
is in the Cartan. Under closure of the Lie bracket, the adjoint action
of the Cartan means that for each $\alpha_{i,j}\neq0$, $\left(\hat{w}_{i,j}+\frac{Q_{i}}{-Q_{j}}\hat{w}_{-j,-i}\right)\in\mathfrak{m}_{\mathbb{C}}.$ \end{proof}
\begin{lem}
\textup{\label{lem:Assume--is}Assume $\varphi\in D$ is CM, nondegenerate,
and odd weight. Let $M$ be a Mumford-Tate group whose Mumford-Tate
domain passes through $\varphi$. Let $v\in\mathfrak{m_{\mathbb{Q}}}.$
If $\pi_{v}$ is trivial, then $\Gamma_{M}$ contains the complete
graph and $\mathfrak{m}=\mathfrak{g}$.}\end{lem}
\begin{proof}
$\Gamma_{v}$ has a single connected component. By construction, $\Gamma_{M}$
has a single connected component. By Lemma \ref{lem:Non-degenerate}
if the vertices $w_{i}$ and $w_{j}$ are connected in $\Gamma_{v}$,
then 
\[
\hat{w}_{i,j}+\frac{Q_{i}}{-Q_{j}}\hat{w}_{-j,-i}\in\mathfrak{m_{\mathbb{C}}}\mbox{\ensuremath{.}}
\]
 Assume vertices $l$ and $m$ are both connected to vertex $k$.
By closure under Lie bracket we have 
\[
\left[\left(w_{l,k}+\frac{Q_{l}}{-Q_{k}}\hat{w_{-k,-l}}\right),\left(w_{k,m}+\frac{Q_{i}}{-Q_{j}}\hat{w_{-m,-k}}\right)\right]=\left(w_{l,m}+\frac{Q_{l}}{-Q_{m}}w_{-m,-l}\right)\in\mathfrak{m}_{\mathbb{C}}.
\]

It follows that there exists $v'\in\mathfrak{m}_{\mathbb{Q}}$ where
$w_{l}$ and $w_{m}$ are connected in $\Gamma_{v'}$ and hence $\Gamma_{M}$
. Proceeding inductively, we conclude from $\Gamma_{M}$ having a
single connected component that it contains the complete graph. Invoking
\ref{lem:Non-degenerate} again, we have that the entire Hodge basis
of $\mathfrak{g}_{\mathbb{C}}$ is contained in $\mathfrak{m}_{\mathbb{C}}$.\end{proof}
\begin{thm}
Assume $\varphi\in D$ be CM and nondegenerate {[}nondegenerate implies
irreducible{]}. Assume $M$ is a Mumford-Tate group whose Mumford-Tate
domain passes through $M$. Let $N\in\mathfrak{m}_{\mathbb{Q}}$ be
nilpotent. If $N^{l-1}\neq0$ with $l>\dim(V)/2$ then $\mathfrak{m}_{\mathbb{C}}=\mathfrak{g}_{\mathbb{C}}$. \end{thm}
\begin{proof}
Follows from Lemma \ref{lem:Assume--is} and Corollary \ref{cor:Trivial parition}.\end{proof}

We now give examples of CM points that are never contained in nilpotent orbits or positive dimension unconstrained MT domains. 
\begin{thm}
Let $D$ be a period domain with odd weight $n>1$, $h^{n,0}=h^{n-1,1}=1$,
and $4\nmid(\dim_{\mathbb{Q}}V)$. If $\varphi$ is irreducible and
CM then $\mathfrak{g}^{1,-1}\oplus\mathfrak{g}^{0,0}\oplus\mathfrak{g}^{-1,1}\cap\mathfrak{g}_{\mathbb{Q}}\subseteq\mathfrak{g}^{0,0}$\end{thm}
\begin{proof}
Set the Hodge-Galois basis as $\left\{ w_{1}...w_{n},w_{-1}...w_{-n}\right\} $
with $w_{1}\in H^{n,0}$ and $w_{2}\in H^{n-1,1}$. Since $r(\varphi)=1$,
there exists a $v\in\left(\mathfrak{g}^{1,-1}\oplus\mathfrak{g}^{0,0}\oplus\mathfrak{g}^{-1,1}\right)\cap\mathfrak{g}_{\mathbb{Q}}$
and $v\notin \mathfrak{g}^{0,0}$. We write $v$ in
the Hodge-Galois basis of $\mathfrak{g}_{\mathbb{C}}$: $v=\sum\alpha_{i,j}\left(\hat{w}_{i,j}+\frac{Q_{i}}{-Q_{j}}\hat{w}_{-j,-i}\right)$.
Since $h^{n,0}=h^{n-1,1}=1$ we have the following fact: 

\begin{equation}
\hat{w}_{1,m}+\frac{Q_{1}}{-Q_{m}}\hat{w}_{-m,-1}\in\mathfrak{g}^{1,-1}\oplus\mathfrak{g}^{0,0}\oplus\mathfrak{g}^{-1,1}\Longrightarrow m=2\mbox{ or }m=1.\label{eq:-3}
\end{equation}

This implies that for $\Gamma_{v}$ vertex $w_{1}$ can only be connected
to $w_{2}$ or itself. Since $v\notin F^{0}(\mathfrak{g}_{\mathbb{C}})$,
$w_{1}$ and $w_{2}$ must be connected. 

Assume $w_{2}$ is connected to $w_{n}$ with $n\neq1$ nor $2$.
Let $g\in Gal(L^{c}/\mathbb{Q})$ be the element that takes $w_{n}$ 
to $w_{1}$. By \ref{eq:-3} it follows that $g$ takes $w_{2}$  to
$w_{2}$. The same reasoning leads us to the contradiction that $g$ 
takes $w_{1}$ to $w_{1}$. It follows that that  $w_{1}$ and $w_{2}$    
are not connected to any other vertices. 

Since the two vertices form a single connected component, $\{w_{1},w_{2}\}$
is an imprimitive domain. By Dodson's imprimitivity theorem, $\{w_{1},w_{2},\bar{w}_{1},\bar{w}_{2}\}$
also generates a imprimitive partition, implying that $4|(\dim_{\mathbb{Q}}L)$,
which is a contradiction(\cite{D} pg. 3).\end{proof}


\begin{thebibliography}{GGK}
\bibitem[A]{A}S. Abdulali, Hodge Structures of CM-Type, Journal of
Ramanujan Mathematical Society 20 (2005), 155-162.

\bibitem[CK]{CK}E. H. Cattani, A. G. Kaplan, Horizontal SL2-Orbits
in Fag Domains, Math. Ann. 235 (1978), 17-35. 

\bibitem[D]{D}B. Dodson. The Structure Of Galois Groups Of CM-Fields,
Transactions Of The American Mathematical Society 283 (1984).

\bibitem[FL]{FL}R. Friedman, R. Laza. Semi-algebraic Horizontal Subvarieties
of Calabi-Yau Type, Duke Mathematical Journal 162 (2013), no. 12,
2077-2148.

\bibitem[GGK]{GGK}M. Green, P. A. Griffiths, and M. Kerr, Mumford-Tate
groups and domains, Annals of Mathematics Studies 183, Princeton University
Press, 2012.

\bibitem[I]{I}A. W. Ingleton. The Rank of Circulant Matrices, Journal
of the London Mathematical Society, vol. s1-31, issue 4, 445-460 1956.

\bibitem[R]{R} K. Ribet, Division fields of Abelian varieties with complex multiplication, Soc. Math. France, 2e Ser.,
Mem. No. 2, 1980, pp. 75-94.

\bibitem[R2]{R2} K. Ribet Generalization of a theorem of Tankeev, Sem. Theorie des Nombres, annee 1981-1982, Exp.
no. 17, 1982


\bibitem[S]{S}W. Schmid, Variations of Hodge Structure: The Singularities
of the Period Mapping, Inventiones Math. 22 (1973) , 211 -319.


\end{thebibliography}
\end{document}